\documentclass[11pt]{article}
  \usepackage{a4wide}
  \usepackage{amsmath}
  \usepackage{amssymb}
  \usepackage{latexsym}
  \usepackage[pdftex]{graphicx}
  \usepackage{subcaption}
  \usepackage{color}
  \usepackage[mathscr]{euscript}

  \newenvironment{proof}{\vspace{1ex}\noindent{\bf Proof.}}{\hspace*{\fill}$\blacksquare$\vspace{1ex}}
  \newenvironment{proofof}[1]{\vspace{1ex}\noindent{\bf Proof of #1.}}{\hspace*{\fill}$\blacksquare$\vspace{1ex}}

  \newtheorem{theorem}{Theorem} 
  \newtheorem{lemma} [theorem] {Lemma}
  \newtheorem{corollary} [theorem] {Corollary}
  \newtheorem{conjecture} [theorem] {Conjecture}
 
\newcommand{\Acal}[0]{\ensuremath{{\mathcal A}}}

\newcommand{\Kcal}[0]{\ensuremath{{\mathcal K}}}

\newcommand{\Ncal}[0]{\ensuremath{{\mathcal N}}}

\newcommand{\Rcal}[0]{\ensuremath{{\mathcal R}}}

\newcommand{\eR}[0]{\ensuremath{ \mathbb R}}

\newcommand{\eN}[0]{\ensuremath{ \mathbb N}}

\newcommand{\norm}[1]{\ensuremath{\|#1\|}}

\newcommand{\Pee}[0]{\ensuremath{{\mathbb P}}}
\newcommand{\Haa}[0]{\ensuremath{{\mathbb H}}}

 \newcommand{\eps}{\varepsilon}

\DeclareMathOperator{\conv}{conv}
\DeclareMathOperator{\vol}{vol}
\DeclareMathOperator{\diam}{diam}

\DeclareMathOperator{\inr}{inr}
\DeclareMathOperator{\supp}{supp}
\definecolor{orange}{RGB}{255,127,0}
\definecolor{pink}{RGB}{255,150,150}

\begin{document}

\title{On the contractibility of random Vietoris--Rips complexes}

\author{Tobias M\"uller\thanks{Bernoulli Institute, Groningen University, The Netherlands. E-mail: {\tt tobias.muller@rug.nl}. 
Partially supported by NWO grants 639.032.529 and 612.001.409.} \and
Mat\v ej Stehl\'ik\thanks{Université de Paris, CNRS, IRIF, F-75006, Paris, France. E-mail: {\tt matej@irif.fr}.
This work was commenced while the author was at Laboratoire G-SCOP, Univ.\ Grenoble Alpes, France.
Partially supported by ANR project GATO (ANR-16-CE40-0009-01) and by LabEx PERSYVAL-Lab (ANR-11-LABX-0025).}}

\date{}

\maketitle

\begin{abstract}
We show that the Vietoris--Rips complex $\Rcal(n,r)$ built over $n$ points sampled at random from a uniformly positive probability 
measure on a convex body $K\subseteq \eR^d$ is a.a.s.~contractible when $r \geq c \left(\frac{\ln n}{n}\right)^{1/d}$
for a certain constant that depends on $K$ and the probability measure used. This answers a question of Kahle~\cite{Kahle11}. 
We also extend the proof to show that if $K$ is a compact, smooth $d$-manifold with boundary -- but not necessarily convex -- then 
$\Rcal(n,r)$ is a.a.s.~homotopy equivalent to $K$ when $c_1 \left(\frac{\ln n}{n}\right)^{1/d} \leq r \leq c_2$
for constants $c_1=c_1(K), c_2=c_2(K)$.
Our proofs expose a connection with the game of cops and robbers.
\end{abstract}

\section{Introduction}

At least in part because of their role as null models in topological data analysis, simplicial complexes built on random points 
in $d$-dimensional Euclidean space have attracted a lot of attention over 
the past decade or so. We refer the reader to the recent survey article~\cite{BobrowskiKahlesurvey} and the references therein 
for more background information and an overview of the results on random geometric complexes.

In this note, $K \subseteq \eR^d$ will either be a convex body (a convex, compact set with nonempty interior) 
or a compact, smooth $d$-manifold  with boundary; and $\nu$ will 
be a probability measure on $\eR^d$ with probability density $f$, whose support is $K$ and which is uniformly positive on $K$. 
(I.e. $\inf_{x\in K} f(x) > 0.$)
We will consider the random Vietoris--Rips complex $\Rcal(n,r)$ constructed by sampling $X_1,\dots, X_n$ i.i.d.~according to $\nu$ and
declaring a subset of these points a simplex if and only if all its pairwise distances are at most $r$.

If $(E_n)_n$ is a sequence of events then we say that $E_n$ holds {\em asymptotically almost surely} (a.a.s.) if
$\Pee(E_n) \to 1$. Our first result is as follows.

\begin{theorem}\label{thm:contract}
If $K\subseteq\eR^d$ is a convex body then there is a constant $c=c(K,\nu)$ such that $\Rcal(n,r_n)$ is a.a.s.~contractible
whenever $r_n \geq c \left(\frac{\ln n}{n}\right)^{1/d}$. 
\end{theorem}

This theorem answers a question of Kahle~\cite{Kahle11} (bottom of page 569), who asked whether 
Theorem~\ref{thm:contract} holds in the case when $K$ has smooth boundary in addition to being 
a convex body and $\nu$ is the uniform distribution on $K$.

If we assume $K$ is smoothly bounded then we can let go of the condition that $K$ be convex and obtain the 
following. Here and in the rest of the paper we use $\simeq$ to denote homotopy equivalence. 

\begin{theorem}\label{thm:smooth}
If $K\subseteq\eR^d$ is a compact, smooth $d$-manifold  with boundary then there are constants $c_1=c_1(K,\nu), c_2=c_2(K)$ such that 
$\Rcal(n,r_n) \simeq K$ a.a.s.~whenever $c_1 \left(\frac{\ln n}{n}\right)^{1/d} \leq r_n \leq c_2$. 
\end{theorem}

Note that the condition $r_n \leq c_2(K)$ is in general necessary. If for instance $r \geq \diam(K)$ then 
$\Rcal(n,r_n)$ will be the complete simplicial complex on $n$ vertices, and in particular contractible, regardless 
of the precise homotopy type of $K$.
For clarity, we emphasize that in both Theorem~\ref{thm:contract} and~\ref{thm:smooth}, the metric used for the construction 
of the Vietorips--Rips complex is 
the Euclidean metric on the ambient space $\eR^d$.


\paragraph{Related work.} A widely studied subcomplex of the Vietoris--Rips complex is the \emph{\v{C}ech complex}, where a set 
of points spans a simplex if and only if the balls of radius $r/2$ around them have a non-empty intersection. 
Niyogi, Smale and Weinberger~\cite{NSW08} obtained a result for random \v{C}ech complexes that is similar to Theorem~\ref{thm:smooth}, 
and Kahle~\cite{Kahle11} proved a statement analogous to Theorem~\ref{thm:contract} for random \v{C}ech complexes under the uniform
probability measure. 
Homological connectivity is a notion closely related to contractibility. 
Results on homological connectivity of random \v{C}ech complexes can be found in~\cite{Bob19+,dKTV19+,IY20,Kahle11}. 

There is a substantial literature on abstract combinatorial models of random simplicial complexes, including the 
Linal--Meshulam model introduced by Linial and Meshulam in~\cite{LinialMeshulam06}, the random $d$-complex introduced by
Meshulam and Wallach in~\cite{MeshulamWallach09}, and the random clique complex introduced by Kahle in~\cite{Kah09}.
The random clique complex is the clique complex (see the next section for the precise definition) 
of the Erd\H{o}s--R\'enyi random graph and thus in a sense
analogous to the random Vietoris--Rips complex, which is the clique complex of the random geometric graph.
The papers of Kahle~\cite{Kah09,Kah14} and Malen~\cite{Mal19+} contain results 
in the same spirit as ours for this model.

\section{Notation and preliminaries\label{sec:prelim}}

Here we list some notations, definitions and results we will use in the proofs.
For $k \in \eN$ a positive integer we write $[k] := \{1,\dots, k\}$.

A {\em simplicial complex} is a pair $\Delta = (V,\Sigma)$ with $V$ a finite set and 
$\Sigma \subseteq 2^V$ closed under taking subsets.
That is, the elements of $\Sigma$ are subsets of $V$ and $\sigma \in \Sigma$ and $\tau \subseteq \sigma$ implies that also
$\tau \in \Sigma$. The elements of $V$ are called vertices and the elements of $\Sigma$ simplicies.
The standard geometric realization of a complex $\Delta = (V,\Sigma)$ is given by 

$$ \norm{\Delta} := \bigcup_{\sigma \in \Sigma} \conv\left( \{ e_i : v_i \in \sigma \} \right), $$

\noindent
where $e_1,\dots,e_n$ denote the standard basis vectors for $\eR^n$ with $n := V$ and we fix some enumeration 
$V = \{ v_1,\dots, v_n\}$.
A topological space $X$ is {\em triangulable} if there exists a simplicial complex $\Delta$ such that 
$X$ is homeomorphic to $\norm{\Delta}$.
We remark that if $K \subseteq \eR^d$ is a convex body then $K$ is homeomorphic to the convex hull of the standard basis 
in $\eR^{d+1}$ and in particular triangulable.
If $K$ is a compact, smooth $d$-manifold with boundary then $K$ is also triangulable by standard results 
in topology (see for instance~\cite{Munkres66}, Theorem 10.6).

If $X, Y$ are two topological spaces and $f, g : X \to Y$ continuous maps then 
$f$ and $g$ are {\em homotopic} (denoted $f\simeq g$) if there exists a 
continuous map $\varphi : X \times [0,1]\to Y$ such that $\varphi(.,0)=f$ and $\varphi(.,1)=g$.
The spaces $X, Y$ are {\em homotopy equivalent} (notation $X \simeq Y$) if there exists
maps $f_1 : X \to Y, f_2 : Y \to X$ such that $f_2 \circ f_1 \simeq \text{id}_{X_1}$ and
$f_1\circ f_2 \simeq \text{id}_{X_2}$.
We say $X$ is {\em contractible} if it is homotopy equivalent to a single point.
If $X$ is a topological space and $\Delta$ a simplicial complex then $X \simeq \Delta$ denotes 
that $X$ and $\norm{\Delta}$ are homotopy equivalent. 
Similarly, for $\Delta_1, \Delta_2$ simplicial complexes $\Delta_1\simeq\Delta_2$ denotes
their geometric realizations are homotopy equivalent. We say $\Delta$ is contractible if
its geometric realization $\norm{\Delta}$ is.

The {\em nerve} of a family of sets $\Acal = (A_i)_{i\in I}$ is the simplicial complex $\Ncal = \Ncal(\Acal)$ with vertex set $I$ 
where a finite set $\sigma \subseteq I$ is a simplex of $\Ncal$ if and only if 
$\bigcap_{i\in \sigma} A_i \neq \emptyset$.
We will use two different versions of the {\em nerve theorem}, stated as Theorems 10.6 and 10.7 in~\cite{Bjorner95}.

\begin{theorem}[Nerve theorem, combinatorial version]\label{thm:nervecomb}
Let $\Delta$ be a simplicial complex and let $(\Delta_i)_{i\in I}$ be a family of subcomplexes
such that $\Delta = \bigcup_{i\in I}\Delta_i$, and 
every nonempty finite intersection $\Delta_{i_1} \cap \dots \cap \Delta_{i_k}$ is contractible.
Then $\Ncal(\Delta) \simeq \Delta$.
\end{theorem}

\begin{theorem}[Nerve theorem, geometric version]\label{thm:nervetop}
Let $X$ be a triangulable space and $\Acal = (A_i)_{i\in I}$ 
either 
a locally finite family of open subsets 
or a finite family of closed subsets, 
such that $X = \bigcup_{i\in I} A_i$.
If every nonempty finite intersection $A_{i_1}\cap \dots \cap A_{i_k}$ is contractible, then $\Ncal(\Acal)\simeq X$. 
\end{theorem}

For $G$ a graph and $v \in V(G)$ a vertex of $G$ we denote by $G\setminus v$ the graph 
with the vertex $v$ and all incident edges removed. By $N_G(v)$ we denote the set of neighbours of $v$ and 
by $N_G[v] := N_G(v) \cup \{v\}$ its {\em closed neighbourhood}.
If no confusion can arise we drop the subscript and simply write $N(v), N[v]$.

If $G$ is a graph, then $\Kcal(G)$ denotes its {\em clique complex}, or {\em flag complex}.
That is, $\Kcal(G)$ has vertex set $V(G)$ and a subset of the vertices is declared a simplex if and only if
it spans a complete subgraph of $G$.
So in particular, $\Rcal(n,r) = \Kcal( G(n,r) )$ where $G(n,r)$ denotes the random geometric graph
with vertices $X_1, \dots, X_n$ and an edge $X_iX_j$ if and only if $\norm{X_i-X_j} \leq r$.

We will use $B(x,r) := \{ y \in \eR^d : \norm{x-y} < r \}$ to denote the open ball of radius $r$ around $x$.
We use $\vol(\cdot)$ to denote the $d$-dimensional volume (Lebesgue measure) and 
we denote by $\pi_d := \vol( B(\underline{0}, 1 ) )$ the volume of the $d$-dimensional unit ball.
For $X \subseteq \eR^d$ we denote its convex hull by $\conv(X)$, its {\em diameter} by $\diam(X) := \sup\{ \norm{x-y} : x,y\in X\}$ and 
we define its {\em inradius} by 

$$ \inr(X) := \sup\{ r : \exists x \text{ such that } B(x,r) \subseteq X \}. $$

\section{Proofs}

\subsection{The proof of Theorem~\ref{thm:contract}}

We start with the following key observation.

\begin{lemma}\label{lem:del}
If $G$ is a graph and $v\neq w \in V(G)$ are such that $N[v] \subseteq N[w]$ 
then $\Kcal( G ) \simeq \Kcal( G \setminus v )$.
\end{lemma}

\begin{proof}
Without loss of generality we can assume $V(G) = [n], v=n, w=n-1$.
We consider the standard geometric realizations $X_1 := \norm{\Kcal(G)}, X_2 := \norm{\Kcal(G\setminus v)}$ where 
we identify $\eR^{n-1}$ with $\eR^{n-1}\times\{0\} \subseteq \eR^n$ so that $X_2 \subseteq X_1$.
Put differently, we have

$$ X_1 = \bigcup_{A \subseteq [n], \atop \text{$A$ a clique in $G$}} \text{conv}(\{ e_i : i \in A\} ), 
\quad  X_2 = \bigcup_{A \subseteq [n]\setminus\{v\}, \atop \text{$A$ a clique in $G$}} \text{conv}(\{ e_i : i \in A\} ), $$

\noindent
with $e_1,\dots,e_n$ the standard basis for $\eR^n$.
%
We need to demonstrate the existence of maps $f_1 : X_1 \to X_2$ and $f_2 : X_2 \to X_1$ such that 
$f_1\circ f_2$ is homotopic to the identity on $X_2$ and $f_2 \circ f_1$ is homotopic to the identity on $X_1$.
We define these maps by

$$ f_1( x_1,\dots, x_n ) := (x_1,\dots,x_{n-2}, x_{n-1}+x_n, 0 ), \quad f_2( x ) := x. $$

\noindent
So $f_2 : X_2 \to X_1$ is simply ``the inclusion map'', but we need to verify that $f_1(x) \in X_2$ for every $x \in X_1$.
To this end, we first note that if $x=(x_1,\dots,x_n)$ satisfies $x_n=0$ then $f_1(x)=x \in X_2 \subseteq X_1$.
Suppose then $x_n \neq 0$ and let $A := \supp(x) = \{ i : x_i \neq 0 \}$ denote the support of $x$.
By definition of $X_1$, we have that $A$ is a clique of $G$ and $\sum_{i\in A} x_i = 1$ and $x_i > 0$ for all $i \in A$.
Setting $y := f_1(x)$ we have that $B := \supp(y) = (A \cup \{n-1\}) \setminus \{n\}$.
By definition of $f_1$ and $A$ and $B$

$$ \sum_{i \in B} y_i = \sum_{i\in A} x_i = 1, $$

\noindent
and $y_i > 0$ for all $i \in B$. So we just need to establish that $B$ is a clique of $G\setminus n$ in order
to verify that $y \in X_2$. In fact, it suffices to show $A \cup \{n-1\}$ is a clique of $G$.
To see this, we remark that $A$ is a clique in $G$ with 

$$ A \subseteq N[n] \subseteq N[n-1], $$

\noindent
using $x_n \neq 0$ for the first inclusion.
In particular each element of $A$ is either a neighour of $n-1$ or $n-1$ itself. 
It follows that $A \cup \{n-1\}$ is a clique indeed.

Since $x_n=0$ whenever $(x_1,\dots,x_n) \in X_2$, we have $f_1 \circ f_2 = \text{id}_{X_2}$.
It remains to see that $f_2 \circ f_1 \simeq \text{id}_{X_1}$. We define $\varphi : X_1 \times [0,1] \to X_1$ via

$$ \varphi( x_1,\dots, x_n, t ) := (x_1, \dots, x_{n-2}, x_{n-1} + t x_n, (1-t) x_n ). $$

The map $\varphi$ satisfies $\varphi(.,0) = \text{id}_{X_1}, \varphi(.,1) = f_2 \circ f_1$ and 
is obviously continuous, but we need to establish that $\varphi(x,t) \in X_1$ for all $t \in [0,1]$ and $x \in X_1$.
To this end, let $x = (x_1,\dots, x_n) \in X_1$ and $t \in [0,1]$ be arbitrary. 
If $x_n = 0$ then $\varphi(x,t) = x$. Let us thus assume $x_n \neq 0$, and again set 
$A := \supp(x), y := \varphi(x,t), B := \supp(y)$.
By definition of $\varphi$ we again have $B \subseteq A \cup \{n-1\}$ and $\sum_{i\in B} y_i = 1$ 
and $y_i > 0$ for all $i \in B$.
Repeating previous arguments, we find that $B$ is a clique of $G$ and hence $y \in X_1$.
\end{proof}

For $x, y\in \eR^d, r > 0$ we define:

\begin{equation}\label{eq:Wdef} 
W(x,y,r) := \{ z \in B(y,\norm{y-x}) : B(z,r) \supseteq B(x,r) \cap B(y,\norm{y-x} ) \}. 
\end{equation}

\begin{figure}[h!]
\begin{center}
 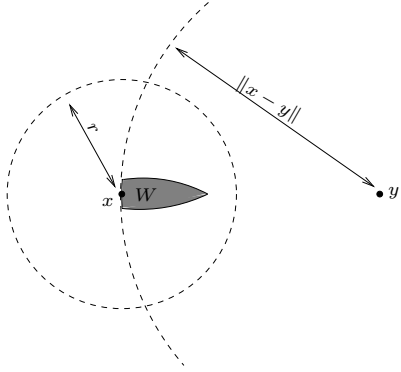
\end{center}
 \caption{Depiction of the set $W(x,y,r)$, in the two-dimensional case.\label{fig:W}}
\end{figure}

See Figure~\ref{fig:W} for a depiction.
The idea behind this definition is as follows. 
Suppose $x,y$ are vertices of a subgraph $H$ of the random geometric graph, and suppose that 
$x$ is the vertex of $H$ that is furthest from $y$. Then any vertex in $z \in W(x,y,r) \cap V(H)$
satisfies $N_H[z] \supseteq N_H[x]$.

\begin{lemma}\label{lem:Wxyr}
For every $\lambda > 0$ there exists a $\delta_1 = \delta_1(\lambda)>0$ such that 
for all $r>0$ and $x, y \in \eR^d$ with $r \leq \norm{x-y} \leq \lambda r$, there is a
$z \in \conv(\{x,y\})$ such that $B(z, \delta_1 r ) \subseteq W(x,y,r)$.
\end{lemma}

\begin{proof}
By applying a suitable dilation and rigid motion if needed, we can assume
without loss of generality that $r=1, x = \underline{0}, y = (y_1, 0, \dots, 0)$ with $1 \leq y_1 \leq \lambda$.

We set 

$$z :=\left(\frac{1}{10\lambda}, 0, \dots, 0\right), \quad 
\delta_1 := \min\left( \frac{1}{10\lambda}, 1 - \left(1-\frac{1}{100\lambda^2}\right)^{1/2} \right). $$

We observe that, since $0 < 1/(10\lambda) < \lambda$, we have $z \in \conv(\{x,y\})$; and by 
the choice of $\delta_1$ we have $B(z,\delta_1) \subseteq B(y, \norm{y-x})$.

Pick an arbitrary $u = (u_1, \dots, u_d) \in B(x,1) \cap B(y, \norm{y-x} )$. 
To complete the proof, we need to show $u \in B(z',1)$ for every $z' \in B(z,\delta_1)$.
We remark that it in fact suffices to show that $\norm{u-z} < \left(1-\frac{1}{100\lambda^2}\right)^{1/2}$.
In order to prove that, we consider two cases.  
We first suppose that $u_1 \geq z_1 = \frac{1}{10\lambda}$. Since $\norm{u-\underline{0}} < 1$, we have
$\sum_{i\geq 2} u_i^2 < 1 - u_1^2$. Hence

$$ \norm{u - z}^2 = (u_1-z_1)^2 + \sum_{i\geq 2} u_i^2 < (u_1-z_1)^2 + 1 - u_1^2 = 
1 + z_1^2 - 2 u_1 z_1 \leq 1 - z_1^2 = 1 - \frac{1}{100\lambda^2}. $$

Let us now suppose that $u_1 < z_1$. Since $u \in B(y,\norm{y-x}) = B(y,\norm{y})$ we have $u_1 > 0$ and 
$\norm{u - y} < \norm{y}$.
Therefore 

$$\sum_{i\geq 2} u_i^2 < y_1^2 - (y_1-u_1)^2 < y_1^2 - (y_1-z_1)^2
= 2 y_1 z_1 - z_1^2  \leq 2\lambda z_1 - z_1^2, $$

\noindent
giving

$$ \norm{u-z}^2 =  (u_1-z_1)^2 + \sum_{i\geq 2} u_i^2 
< z_1^2 + 2\lambda z_1 - z_1^2 = 2 \lambda z_1 = \frac15 <  1 - \frac{1}{100\lambda^2}, 
$$

\noindent
using $0<u_1<z_1$ in the first inequality and $\lambda\geq 1$ in the second inequality. 
\end{proof}

The only probabilistic ingredient we will need is the lemma below on the probability that the balls of radius $r$ around the 
random points $X_1,\dots, X_n$ cover $K$. 

\begin{lemma}\label{lem:cover}
Suppose $K\subseteq\eR^d$ is a convex body. 
There is a $\delta_2 = \delta_2(K,\nu)$ such that if $r_n \geq \delta_2 \cdot \left(\frac{\ln n}{n}\right)^{1/d}$ then, a.a.s., 
$K \subseteq \bigcup_{i=1}^n B(X_i,r_n)$.
\end{lemma}

We will employ the following observation in the proof of Lemma~\ref{lem:cover} that will be of use to us later on as well.

\begin{lemma}\label{lem:inr}
 For every $\lambda > 0$ there exists a $\delta_3=\delta_3(\lambda)>0$ such that, for every convex $X \subseteq \eR^d$ 
 with $\diam(X) \leq \lambda \cdot \inr(X)$, for every 
 $0<r\leq \diam(X)$ and $x \in X$ there is a $z$ such that 
 
 $$ B(z,\delta_3 r) \subseteq B(x,r) \cap X. $$
 
\end{lemma}

\begin{proof}
Applying a suitable dilation if needed, we can assume without loss of generality $\inr(X) = 1$.
By assumption, there exists $y \in X$ such that $B(y,1/2) \subseteq X$.
Applying a suitable translation if needed, we can assume without loss of generality that 
$y=\underline{0}$ is the origin.
By convexity, the set 

$$ B_\mu := 
\{ \mu z + (1-\mu) x : z \in B(\underline{0},1/2) \}. $$

\noindent
is contained in $X$ for every $0\leq \mu \leq 1$.
On the other hand $B_\mu$ is a ball of radius $\mu/2$ centred at the point $(1-\mu) x$.
In particular 

$$ B_\mu \subseteq B(x, \mu/2+ \norm{x-(1-\mu)x}) = B(x, \mu \cdot (1/2 + \norm{x})) \subseteq
B( x, \mu \cdot (1/2+\lambda) ). $$

So, if we set $\mu := r / (1/2+\lambda)$
then $B_\mu \subseteq B(x,r) \cap X$ is a ball of radius $\delta_3 r$ where $\delta_3=\delta_3(\lambda):= 1/(1+2\lambda)$.
\end{proof}

\begin{proofof}{Lemma~\ref{lem:cover}}
We set $r := \delta_2 \cdot \left(\ln n/n\right)^{1/d}$ where the constant $\delta_2$ will be determined
in the course of the proof. We write $s := r/4$ for convenience.
Next, we pick $x_1, \dots, x_N \in K$ such that $B(x_1, s), \dots, B(x_N,s)$ are disjoint and
$K \subseteq \bigcup_{i=1}^N B(x_i, 2s)$.
(Such $x_1,\dots, x_N$ can for instance be found by ``greedily'' constructing a maximal
packing of balls of radius $s$ with centers in $K$.)
By Lemma~\ref{lem:inr}, for each $i$, $B(x_i,s) \cap K$ contains a ball of radius $\delta_3 \cdot s$ 
with $\delta_3 = \delta_3( \diam(K)/\inr(K) )$.
Hence, writing $\nu_{\min} := \inf_{x\in K} f(x)$, 
we have, for each $i=1,\dots,N$:

\begin{equation}\label{eq:nuballLB} 
\nu( B(x_i,s) ) \geq \nu_{\min} \cdot \pi_d \cdot (\delta_3 s)^d
= \nu_{\min} \cdot \pi_d \cdot (\delta_2/4)^d \cdot \delta_3^d \cdot \left(\frac{\ln n}{n}\right). 
\end{equation}

\noindent
Since the balls $B(x_i, s)$ are disjoint, we have $1 \geq \sum_i \nu( B(x_i,s) )$.
Combining this with~\eqref{eq:nuballLB} gives $N = O( n/\ln n )$.
The inequality~\eqref{eq:nuballLB} also implies that 

$$ \begin{array}{rcl} 
\Pee( B(x_i, 2s) \cap \{X_1,\dots,X_n\} = \emptyset )
& = & \left(1-\nu(B(x_i, 2s))\right)^n  \\
& \leq & \left(1 - \nu_{\min} \cdot \pi_d \cdot \delta_2^d \cdot \delta_3^d \cdot \left(\frac{\ln n}{n}\right) \right)^n \\
& \leq & \exp\left[ - \ln n \cdot  \nu_{\min} \cdot \pi_d \cdot (\delta_2/4)^d \cdot \delta_3^d \right], 
\end{array} $$

\noindent
for each $i$. 
As $K \subseteq \bigcup_{i=1}^N B(x_i, 2s)$ and $r=4s$ we see that

$$ \begin{array}{rcl} \Pee\left( \text{$K$ is not covered by $\bigcup_{i=1}^n B(x_i, r)$}\right)
& \leq & \Pee(B(x_i, 2s) \cap \{X_1,\dots,X_n\} = \emptyset \text{ for some $1\leq i\leq N$} ) \\
& \leq & N \cdot \exp\left[ - \ln n \cdot  \nu_{\min} \cdot \pi_d \cdot (\delta_2/4)^d \cdot \delta_3^d \right] \\
& = & o(1), 
\end{array} $$

\noindent
where the last line holds provided we chose $\delta_2 > 4^d /(\nu_{\min} \cdot \pi_d \cdot \delta_3^d)$.
\end{proofof}

\begin{lemma}\label{lem:motor}
Let $K \subseteq \eR^d$ be a convex body. There exists a constant $\delta_4=\delta_4(K)$ such that,
for every $0<r\leq\diam(K)$, there exists a finite family of sets 
$A_1,\dots, A_N \subseteq K$ such that 
\begin{enumerate}
 \item\label{itm:motor1} Each $A_i$ is open in the relative topology of $K$, and;
 \item\label{itm:motor2} For every $S \subseteq K$ with $\diam(S) \leq r$ there is a
 $1\leq i\leq N$ such that $S \subseteq A_i$, and;
 \item\label{itm:motor3} For every $I \subseteq [N]$, if $\bigcap_{i\in I} A_i$ is nonempty then it is contractible, and;
 \item\label{itm:motor4} For every $I \subseteq [N]$, if $\bigcap_{i\in I} A_i$ is nonempty then it contains 
 a ball of radius $\delta_4\cdot r$, and;
 \item\label{itm:motor5} For every $I \subseteq [N]$, if $x, y \in \bigcap_{i\in I} A_i$ 
 satisfy $\norm{x-y} \geq r$ then $W(x,y,r) \cap \left( \bigcap_{i\in I} A_i \right)$ contains a ball of radius $\delta_4\cdot r$.
\end{enumerate}
\end{lemma}

\begin{proof}
We let $x_1, \dots, x_N \in K$ be such that $B(x_1, r), \dots, B(x_N,r)$ are disjoint and
$K \subseteq \bigcup_{i=1}^N B(x_i, 2r)$.
We are going to set $A_i := B(x_i,s_i) \cap K$ where 
the radii $3r \leq s_i \leq 4r$ will be determined via an iterative procedure.
Initially we set $s_1 = \dots = s_N = 3r$.
We let $\eps>0$ be a parameter, to be chosen more precisely later on.

If there is a set of indices $I \subseteq [N]$ such that
$K \cap \bigcap_{i\in I} B(x_i,s_i)$ is nonempty, but there is no
$x \in K \cap \bigcap_{i\in I} B(x_i,s_i)$ such that $B(x,\eps r) \subseteq \bigcap_{i\in I} B(x_i,s_i)$,
then we increase $s_i$ by $\eps r$ for each $i \in I$.
(Note that this will ensure that there is an $x \in K$ with $B(x,\eps r) \subseteq \bigcap_{i\in I} B(x_i,s_i)$.)
We keep repeating this step until this is no longer possible.

Since the operation is applied at most once to each $I \subseteq [N]$ the procedure finishes after a finite number 
of operations.
It may however not be obvious that there is a choice of $\eps>0$ for which this procedure will finish 
in a situation where $s_1,\dots,s_N \leq 4r$.
To see this we set

$$ I_i := \{ j : B(x_i,4r) \cap B(x_j,4r) \neq \emptyset \}, \quad m_i := |I_i|. $$

\noindent
If $B(x_i,4r) \cap B(x_j,4r) \neq \emptyset$ then $B(x_j,r) \subseteq B(x_i,9r)$.  
As the balls $B(x_1,r), \dots, B(x_N,r)$ are disjoint this gives

$$ \pi_d \cdot (9r)^d = \vol( B(x_i, 9r ) ) \geq \sum_{j\in I_i} \vol( B(x_j,r) ) = m_i \cdot 
\pi_d\cdot r^d. $$

\noindent
It follows that

$$ m_i \leq 9^d. $$

\noindent
The number of sets of indices $I$ with $i \in I$ and $\bigcap_{j\in I} B(x_j,4r) \neq \emptyset$ is therefore at most 
$2^{m_i} \leq 2^{9^d}$. 
Having chosen $\eps := \frac{1}{100} \cdot 2^{-9^d}$ we see that no radius $s_i$ will ever 
exceed $4r$.

We need to prove that, for a suitable choice of the constant $\delta_4$, the constructed 
sets $A_1 := B(x_1,s_i) \cap K, \dots, A_N := B(x_N, s_N) \cap K$
satisfy the properties claimed by the lemma. 
That~\ref{itm:motor1} holds is obvious and that~\ref{itm:motor3} 
holds follows immediately from the fact that each $A_i$, and hence also 
each nonempty intersection $\bigcap_{i\in I} A_i$, is convex.
To see that~\ref{itm:motor2} holds, let $S \subseteq K$ with $\diam(S) \leq r$ be arbitrary and 
fix $s \in S$. Since $K \subseteq B(x_1,2r) \cup \dots \cup B(x_N, 2r)$ there is 
some $i$ such that $s \in B(x_i,2r)$. But then 

$$ S \subseteq K \cap B(s,r) \subseteq K \cap B(x_i, 3r) \subseteq K \cap B(x_i, s_i ) = A_i. $$

It remains to establish~\ref{itm:motor4},~\ref{itm:motor5} for a suitable choice of the constant $\delta_4$.
We let $\delta_3=\delta_3(\diam(K)/\inr(K))$ be as provided by Lemma~\ref{lem:inr}.
If $I$ is such that $\bigcap_{i\in I} A_i = K \cap \bigcap_{i\in I} B(x_i, s_i)$ is non-empty
then by construction there exists $x \in K$ such that $B(x,\eps r) \subseteq \bigcap_{i\in I} B(x_i, s_i)$.
Applying Lemma~\ref{lem:inr} we find a ball $B(z, \delta_3 \eps r) \subseteq B(x,\eps r) \cap K
\subseteq \bigcap_{i\in I} A_i$.

Now suppose $x, y \in \bigcap_{i\in I} A_i$ with $\norm{x-y} \geq r$, and let 
$\delta_1 = \delta_1(8/\delta_3\eps)$ be as provided by Lemma~\ref{lem:Wxyr}.
(Observe that $\diam\left(\bigcap_{i\in I} A_i\right) \leq 8r$ and 
$\inr\left(\bigcap_{i\in I} A_i\right) \geq \delta_3 \eps r$.)
It follows from Lemma~\ref{lem:Wxyr} that there exists a 
ball $B(z,\delta_1 r) \subseteq W(x,y,r)$ with $z \in \conv(\{x,y\}) \subseteq \bigcap_{i\in I} A_i$.
(The inclusion holding by convexity.)
Applying Lemma~\ref{lem:inr} again, we find a ball 

$$ B(z', \delta_3 \delta_1 r ) \subseteq B(z, \delta_1 r) \cap \bigcap_{i\in I} A_i
\subseteq W(x,y,r) \cap \bigcap_{i\in I} A_i. $$

We conclude that~\ref{itm:motor4} and~\ref{itm:motor5} hold with $\delta_4 := \min(\delta_3\delta_1, \delta_3\eps )$.
\end{proof}

\begin{proofof}{Theorem~\ref{thm:contract}}
By Lemma~\ref{lem:cover}, there is a $\delta_2=\delta_2(K)$ such that, a.a.s., every ball $B(x, \delta_2\left(\ln n/n\right)^{1/d} )$
with $x\in K$ contains at least one of the random points 
$X_1,\dots, X_n$.
We set $c := \delta_2/\delta_4$ with $\delta_4=\delta_4(K)$ as provided by Lemma~\ref{lem:motor}.
Let $r \geq c\left(\ln n/n\right)^{1/d}$ be arbitrary.
Note that if $r \geq \diam(K)$ then $\Rcal(n,r)$ is the complete simplicial complex on $n$ vertices and we are trivially done.
We thus assume $r < \diam(K)$ for the remainder of the proof, and we let $A_1,\dots, A_N$ be as provided
by Lemma~\ref{lem:motor}.
By the geometric version of the nerve theorem and items~\ref{itm:motor1} and~\ref{itm:motor3} of Lemma~\ref{lem:motor}, we have 
$\Ncal( (A_i)_{i=1,\dots, N} ) \simeq K$.

Let $\Delta_i$ denote subcomplex of $\Rcal(n,r)$ induced by the points $\{X_1,\dots,X_n\} \cap A_i$ that fall in $A_i$.
By the combinatorial version of the nerve theorem, to finish the proof of Theorem~\ref{thm:contract} 
it suffices to show that, a.a.s, {\bf a)} $\bigcup_i \Delta_i = \Rcal(n,r)$,  
{\bf b)} every nonempty finite intersection $\bigcap_{i\in I} \Delta_i$ is contractible, and 
{\bf c)} $\bigcap_{i\in I} A_i \neq \emptyset$ if and only if $\bigcap_{i\in I} \Delta_i \neq \emptyset$. 
(Note that {\bf c)} in fact states that the nerves $\Ncal( (A_i)_{i=1,\dots, N} )$ and $\Ncal( (\Delta_i)_{i=1,\dots, N} )$ are identical.)

The demand {\bf a)} immediately follows from item~\ref{itm:motor2} of Lemma~\ref{lem:motor}. 
That $\bigcap_{i\in I} \Delta_i \neq \emptyset$ implies $\bigcap_{i\in I} A_i \neq \emptyset$ is immediate from the definition of
$\Delta_i$.
If $\bigcap_{i\in I} A_i \neq \emptyset$ then in fact there is a ball 

$$ B(x, \delta_2\left(\ln n/n\right)^{1/d} ) \subseteq B(x,\delta_4 r )\subseteq \bigcap_{i\in I} A_i. $$ 

\noindent
As we have seen above, a.a.s., every such disk contains at least one of the random points
$X_1,\dots, X_n$. But then this random point is a vertex of $\bigcap_{i\in I} \Delta_i$ by definition of the subcomplexes $\Delta_i$.
This establishes that {\bf c)} holds.

It remains to see why {\bf b)} holds. Suppose that $\bigcap_{i\in I} \Delta_i$ is nonempty, 
and let $X_{j_1}, \dots, X_{j_k}$ denote its vertices.
We can assume without loss of generality that the points are labelled by non-decreasing distance to $X_{j_1}$, 
i.e.~

$$\norm{X_{j_2}-X_{j_1}} \leq \dots \leq \norm{X_{j_k}-X_{j_1}}. $$

\noindent
By item~\ref{itm:motor5} of Lemma~\ref{lem:motor}, for each $2\leq \ell \leq k$, we have either 
$\norm{X_{j_\ell} - X_{j_1}} < r$ or there is a ball $B(z,\delta_4 r) \subseteq W(X_{j_\ell},X_{j_1},r) \cap 
\bigcap_{i\in I} A_i$.
By the choice of $c$, a.a.s., any such ball contains one of the random points $X_1,\dots, X_n$.
In particular, $B(z,\delta_4 r)$ contains one of $X_{j_1}, \dots, X_{j_{\ell-1}}$.
In other words, for each $2\leq \ell \leq k$, we have either 
$\norm{X_{j_\ell} - X_{j_1}} < r$ or there exists a $2\leq m(\ell) < \ell$ such that 
$X_{j_{m(\ell)}} \in W(X_{j_\ell}, X_{j_1}, r)$.
Writing $G_\ell$ for the subgraph of the random geometric graph induced by $X_{j_1}, \dots, X_{j_\ell}$, we see that 

$$ N_{G_\ell}[ X_{j_\ell} ] \subseteq N_{G_\ell}[ X_{j_{m(\ell)}} ], $$

\noindent
where we set $m(\ell) := 1$ when $\norm{X_{j_\ell}-X_{j_1}} < r$.
By repeated applications of Lemma~\ref{lem:del} we now derive

$$\bigcap_{i\in I} \Delta_i = \Kcal(G_k) \simeq \Kcal(G_{k-1}) \simeq \dots \simeq \Kcal( G_1 ). $$

\noindent
So $\bigcap_{i\in I} \Delta_i$ is indeed contractible.
\end{proofof}

\subsection{The proof of Theorem~\ref{thm:smooth}}

Much of the proof of Theorem~\ref{thm:contract} can be reused. 
The next two lemmas represent the main adaptations that need to be made.

\begin{lemma}\label{lem:inrsmooth}
Let $K \subseteq \eR^d$ be a compact, smooth $d$-manifold with boundary.
There exist constants $\delta_3=\delta_3(K), \rho_1 = \rho_1(K)>0$ such that
for every $x \in K$ and $0<r\leq \rho_1$ there is a ball $B(z,\delta_3 r) \subseteq B(x,r) \cap K$.
\end{lemma}

\begin{lemma}\label{lem:locallystarshaped}
Let $K \subseteq \eR^d$ be a compact, smooth $d$-manifold with boundary.
For every $\lambda>0$ there exists a $\rho_2 = \rho_2(K,\lambda)>0$ such that
if $x \in K$ and $B(y,r) \subseteq K$ and $\norm{x-y} \leq \lambda r$ for some $0<r<\rho_2$ then 
$\conv(\{x,y\}) \subseteq K$.
\end{lemma}

Before proceeding with the proof of these two lemmas, we first explain 
how they can be used to adapt the proof of Theorem~\ref{thm:contract} to obtain a proof of Theorem~\ref{thm:smooth}.
We observe that the proof of Lemma~\ref{lem:cover} carries over verbatim if we substitute the use of Lemma~\ref{lem:inr}
with Lemma~\ref{lem:inrsmooth}.

\begin{corollary}\label{cor:coversmooth}
Suppose $K\subseteq\eR^d$ is a compact, smooth $d$-manifold with boundary. 
There is a $\delta_2 = \delta_2(K,\nu)$ such that if $r_n \geq \delta_2 \cdot \left(\frac{\ln n}{n}\right)^{1/d}$ then, a.a.s., 
$K \subseteq \bigcup_{i=1}^n B(X_i,r_n)$.
\end{corollary}

With some minor changes to the statement and its proof, we obtain the following variant of Lemma~\ref{lem:motor}.

\begin{lemma}\label{lem:motorsmooth}
Let $K \subseteq \eR^d$ be a compact, smooth $d$-manifold with boundary. 
There exist constants $\delta_4=\delta_4(K), 
\delta_5=\delta_5(K), \rho_3=\rho_3(K)$ such that,
for every $0<r\leq\rho_3$, there exists a finite family of sets 
$A_1,\dots, A_N \subseteq K$ such that 
\begin{enumerate}
 \item Each $A_i$ is open in the relative topology of $K$, and;
 \item For every $S \subseteq K$ with $\diam(S) \leq r$ there is an 
 $1\leq i\leq N$ such that $S \subseteq A_i$, and;
 \item For every $I \subseteq [N]$, if $\bigcap_{i\in I} A_i$ is nonempty then it is contractible, and;
 \item For every $I \subseteq [N]$, if $\bigcap_{i\in I} A_i$ is nonempty then it contains 
 a ball of radius $\delta_4\cdot r$, and;
 \item For every $I \subseteq [N]$, if $x, y \in \bigcap_{i\in I} A_i$ 
 satisfy $\norm{x-y} \geq r$ and $B(y,(\delta_4/2)\cdot r) \subseteq \bigcap_{i\in I} A_i$ then 
 $W(x,y,r) \cap \left( \bigcap_{i\in I} A_i \right)$ contains a ball of radius $\delta_5\cdot r$.
\end{enumerate}
\end{lemma}

\begin{proofof}{Lemma~\ref{lem:motorsmooth} assuming Lemmas~\ref{lem:inrsmooth} and~\ref{lem:locallystarshaped}}
The proof is largely the same as the proof of Lemma~\ref{lem:motor}. The construction of $A_1,\dots, A_N$ carries over unaltered, 
as do the proofs of~\ref{itm:motor1},~\ref{itm:motor2} and~\ref{itm:motor4} except that we substitute the use of 
Lemma~\ref{lem:inr} with Lemma~\ref{lem:inrsmooth} and have to assume the corresponding upper bounds on $r$.
For part~\ref{itm:motor3} we can no longer use convexity. Instead, we rely on Lemma~\ref{lem:locallystarshaped}, as follows.
Suppose $\bigcap_{i\in I} A_i$ is non-empty. We have already established that there is a ball $B(x,\delta_4 \cdot r) \subseteq 
\bigcap_{i\in I} A_i$. 
Since $\diam\left( \bigcap_{i\in I} A_i \right) \leq 8r$, under the assumption that $r \leq \rho_2(K, 8/\delta_4 )$, we know 
that for all $y \in \bigcap_{i\in I} A_i$ the line segment $\conv(\{x,y\})$ is contained in $K$.
Since, for each $i$, we have defined $A_i := B(x_i,s_i) \cap K$ and $B(x_i,s_i)$ is convex, it follows that 

$$\conv(\{x,y\}) \subseteq \bigcap_{i\in I} A_i. $$

\noindent
In other words, for each $y \in \bigcap_{i\in I} A_i$, the line segment between $y$ and $x$ 
is completely contained in $\bigcap_{i\in I} A_i$. So $\bigcap_{i\in I} A_i$ is star-shaped, and in particular contractible.

To see that~\ref{itm:motor5} holds, we suppose that $B(y,(\delta_4/2)\cdot r) \subseteq \bigcap_{i\in I} A_i$ and
$x \in \bigcap_{i\in I} A_i$.
For each $z\in B(y,(\delta_4/4)\cdot r)$ we of course have 

$$ B(z,(\delta_4/4) \cdot r ) \subseteq B(y,(\delta_4/2)\cdot r) \subseteq K. $$

\noindent
Therefore, applying Lemma~\ref{lem:locallystarshaped}, provided $r \leq \rho_2(K, 16/\delta_4)$, we have that
$\conv(\{z,x\}) \subseteq K$ for all $z\in B(y,(\delta_4/4) \cdot r )$. 
As $A_i = B(x_i,s_i) \cap K$ and balls are convex, it follows that 
in fact $\conv(\{z,x\}) \subseteq  \bigcap_{i\in I} A_i$. Hence, also

$$ X := \conv\left(B(y,(\delta_4/4) \cdot r ) \cup \{x\} \right) \subseteq \bigcap_{i\in I} A_i. $$

Applying Lemma~\ref{lem:Wxyr}, there is a $z \in \conv(\{x,y\})$ such that 
$B(z,\delta_1 r) \subseteq W(x,y,r)$ where $\delta_1 = \delta_1(8)$.
We now apply Lemma~\ref{lem:inr} to $X$. Note that $\diam(X) \leq 8r, \inr(X) \geq \delta_4/4$.
Hence, taking $\delta_3 = \delta_3(32/\delta_4)$ as provided by Lemma~\ref{lem:inr}, we have

$$ B(z,\delta_3\delta_1 r ) \subseteq X \cap W(x,y,r) \subseteq \left( \bigcap_{i\in I} A_i \right) \cap W(x,y,r). $$

\noindent
In other words, we have established~\ref{itm:motor5} for $\delta_5 := \delta_1\delta_3$.
\end{proofof}

\begin{proofof}{Theorem~\ref{thm:smooth} assuming Lemmas~\ref{lem:inrsmooth} and~\ref{lem:locallystarshaped}}
The proof is largely the same as the proof of Theorem~\ref{thm:contract}.
We assume $c_1 \cdot \left(\ln n/n\right)^{1/d} \leq r \leq c_2$ where the choice of the constants 
will be determined in the course of the proof.
Having chosen $c_2$ appropriately we can apply Lemma~\ref{lem:motorsmooth} to obtain $A_1,\dots, A_N$.
We again let $\Delta_i$ denote subcomplex of $\Rcal(n,r)$ induced by the points $\{X_1,\dots,X_n\} \cap A_i$.
It again suffices to show that, a.a.s, {\bf a)} $\bigcup_i \Delta_i = \Rcal(n,r)$,  
{\bf b)} every nonempty finite intersection $\bigcap_{i\in I} \Delta_i$ is contractible, and 
{\bf c)} $\bigcap_{i\in I} A_i \neq \emptyset$ if and only if $\bigcap_{i\in I} \Delta_i \neq \emptyset$. 

The demand {\bf a)} immediately follows from item~\ref{itm:motor2} of Lemma~\ref{lem:motorsmooth}, and 
that {\bf c)} holds a.a.s.~follows in the same way as in the proof of Theorem~\ref{thm:contract}
(assuming the constant $c_1$ was chosen sufficiently large) except that we use Corollary~\ref{cor:coversmooth}
in place of Lemma~\ref{lem:cover}.

The proof of {\bf b)} needs slightly more adaptation. Suppose again that for some $I \subseteq [N]$ the 
intersection $\bigcap_{i\in I} \Delta_i$ is nonempty and let $X_{j_1}, \dots, X_{j_k}$ denote its vertices.
Since $\bigcap_{i\in I} A_i \supseteq B(z,\delta_4 r)$ for some $z$, having chosen $c_1$ appropriately
large, we can assume without loss of generality that $X_{j_1} \in B(z,(\delta_4/2) r)$.
We also assume, without loss of generality, that the points $X_{j_1}, \dots, X_{j_k}$ are 
labelled by non-decreasing distance to $X_{j_1}$.

That 

$$ X_{j_1} \in B(z,(\delta_4/2) r) \subseteq B(z,\delta_4 r) \subseteq \bigcap_{i\in I} A_i, $$ 

\noindent
of course implies that 

$$ B(X_{j_1}, (\delta_4/2) r ) \subseteq \bigcap_{i\in I} A_i. $$

\noindent
Applying part~\ref{itm:motor5} of Lemma~\ref{lem:motorsmooth}, for each $2\leq \ell \leq k$
we have that either $\norm{X_{j_\ell}-X_{j_1}} \leq r$ or 
$W( X_{j_\ell}, X_{j_1}, r ) \cap \bigcap_{i\in I} A_i$ contains a ball of radius $\delta_5 r$.
Having chosen $c_1$ sufficiently large, a.a.s., any such ball contains one of the 
random points $X_1,\dots, X_n$. 
In other words, for each $2\leq \ell \leq k$, we have either 
$\norm{X_{j_\ell} - X_{j_1}} < r$ or there is an $2\leq m(\ell) < \ell$ such that 
$X_{j_{m(\ell)}} \in W(X_{j_\ell}, X_{j_1}, r)$.
We can thus conclude {\bf b)} holds in the same way we did in the proof of Theorem~\ref{thm:contract}.
\end{proofof}

It remains to prove Lemmas~\ref{lem:inrsmooth} and~\ref{lem:locallystarshaped}.
That $K$ is a smooth $d$-manifold with boundary means that 
for every $x \in K$ there are open sets $O, U \subseteq \eR^d$ with $x\in O$
and a diffeomorphism (a smooth bijection whose inverse is also smooth) $\varphi : O \to U$
such that $\varphi[ O \cap K ] = U \cap \Haa^d$, where $\Haa^d := [0,\infty) \times \eR^{d-1}$.
If $K$ is also compact we have in addition:

\begin{lemma}\label{lem:boundedderivatives}
Let $K \subseteq \eR^d$ be a compact, smooth $d$-manifold with boundary.
There exist $\rho_4 = \rho_4(K) > 0$ and $\delta_6=\delta_6(K)$ such that, for every $x\in K$ there
exist open sets $O_x, U_x \subseteq \eR^d$ and a diffeomorphism $\varphi_x : O_x \to U_x$, such that 
\begin{enumerate}
 \item $\varphi_x[O_x \cap K] = U_x \cap \Haa^d$, $B(x,\rho_4) \subseteq O_x$, $B(\varphi_x(x), \rho_4 ) \subseteq U_x$, and;
 \item The absolute values of the partial derivatives of $\varphi_x$ of order one and two are bounded by $\delta_6$ on $O_x$, and;
 \item The absolute values of the partial derivatives of $\varphi_x^{-1}$ of order one and two are bounded by $\delta_6$ on $U_x$.
\end{enumerate}
\end{lemma}

\begin{proof}
For each $x \in K$ there are open sets $V_x, W_x \subseteq \eR^d$ with $x \in V_x$ and a diffeomorphism 
$\psi_x : V_x \to W_x$ such that $\psi_x[V_x\cap K] = W_x \cap \Haa^d$.
We can assume without loss of generality that, for each $x \in K$, there exists a $c_x < \infty$ such that 
all partial derivatives of $\psi_x, \psi_x^{-1}$ of orders up to two are 
bounded in absolute value by $c_x$.
(Switching to a smaller open subset $V_x' \subseteq V_x$ if needed.)

For each $x \in K$ there exists an $s_x > 0$ such that $B(x,s_x) \subseteq V_x$ and $B(\psi_x(x), s_x ) \subseteq W_x$.
Let us write 

$$ A_x := B(x,s_x/2) \cap \psi_x^{-1}[ B(\psi_x(x), s_x/2 ) ]. $$

Since the sets $A_x$ are open and $K$ is compact, there are $x_1,\dots,x_N$ such that 
$K \subseteq \bigcup_{i=1}^N A_{x_i}$.
We set 

$$ \rho_4 := \frac12 \cdot \left( \min_{i=1,\dots,N} s_{x_i} \right), \quad \quad c := \max_{i=1,\dots,N} c_{x_i}, $$ 

\noindent
and for each $x \in K$ we fix an $i$ such that $x \in A_{x_i}$ and set:

$$ O_x := V_{x_i}, \quad U_x := W_{x_i}, \quad \varphi_x := \psi_{x_i}. $$

Observe that $x \in A_{x_i}$ implies that 

$$B(x, \rho_4 ) \subseteq B(x_i, 2\rho_4) \subseteq B(x_i, s_x ) \subseteq V_{x_i} = O_x, $$

\noindent
and $\psi_{x_i}(x) \in B(\psi_{x_i}(x_i), s_{x_i}/2)$ so that also

$$ B(\varphi_x(x), \rho_4 ) = B( \psi_{x_i}(x), \rho_4 ) \subseteq 
B(\psi_{x_i}(x_i), 2\rho_4) \subseteq B(\psi_{x_i}(x_i), s_{x_i}) \subseteq W_{x_i} = U_x. $$

\end{proof}

\begin{corollary}\label{cor:ballcor}
Let $K \subseteq \eR^d$ be a compact, smooth $d$-manifold with boundary.
There exist $\rho_5=\rho_5(K), \delta_7=\delta_7(K)>0$ such that, for every $x\in K$ and $0<r<\rho_5$:
\begin{enumerate}
 \item\label{itm:ballcor1} For every $y \in B(x,\rho_5)$ we have $B( \varphi_x(y), \delta_7 r ) \subseteq \varphi_x[ B(y,r) ]$, and;
 \item\label{itm:ballcor2} For every $y \in B(\varphi_x(x), \rho_5)$ we have $B(\varphi_x^{-1}(y), \delta_7 r ) \subseteq \varphi_x^{-1}[ B(y,r) ]$.
\end{enumerate}
where $\varphi_x$ is as in Lemma~\ref{lem:boundedderivatives}.
\end{corollary}

\begin{proof}
It follows from Lemma~\ref{lem:boundedderivatives}, using for instance Theorem 9.19 of~\cite{Rudin76}, that 
for every $x\in K$ and $y,z \in B(x,r_0)$ and $y',z' \in B(\varphi_x(x), r_0)$ we have 

$$ \norm{\varphi_x(y) - \varphi_x(z)} \leq \delta_6\sqrt{d} \cdot \norm{y-z}, \quad \quad 
\norm{y'-z'} \leq \delta_6\sqrt{d} \cdot \norm{\varphi_x^{-1}(y') - \varphi_x^{-1}(z')}. $$

\noindent
Hence

$$ \varphi_x^{-1}[ B(\varphi_x(y), \delta_7 r ) ] \subseteq 
B( \varphi_x^{-1}(\varphi_x(y)), \delta_6\sqrt{d} \cdot \delta_7 r ) = B(y,\delta_6\sqrt{d} \cdot \delta_7 r) \subseteq B(y,r), $$

\noindent 
the first inclusion holding provided $r, \norm{y-x} < \rho_5 := \rho_4 / 2$ and
the last inclusion holding provided $\delta_7 < \min(1/\delta_6\sqrt{d}, 1)$.
In other words, with this choice of $\delta_7, \rho_5$ we have:

$$ B(\varphi_x(y), \delta_7 r ) \subseteq \varphi_x[ B(y,r) ], $$

\noindent
establishing~\ref{itm:ballcor1}.
The proof of~\ref{itm:ballcor2} is completely analogous.
\end{proof}

\begin{proofof}{Lemma~\ref{lem:inrsmooth}}
Let $x \in K$ be arbitrary and $r < \rho_5(K)$ with $\rho_5$ as provided by Corollary~\ref{cor:ballcor}.
Applying Corollary~\ref{cor:ballcor}, we have

$$ \varphi_x[ B(x,r) ] \supseteq B( \varphi_x(x), \delta_7r ). $$

\noindent
Since $\varphi_x(x) \in \Haa^d$, there is a ball 

$$ B( z, (\delta_7/2) \cdot r) \subseteq B( \varphi_x(x), \delta_7r ) \cap \Haa^d, $$

\noindent
of radius $(\delta_7/2) \cdot r$ contained in $B( \varphi_x(x), \delta_7r ) \cap \Haa^d$.
Applying Corollary~\ref{cor:ballcor} once again, we have 

$$ B(x,r) \cap K \supseteq \varphi_x^{-1}[ B( z, (\delta_7/2)\cdot r) ]  \supseteq B( \varphi_x^{-1}(z), (\delta_7^2/2)\cdot r ). $$

\noindent
We can conclude Lemma~\ref{lem:inrsmooth} holds with $\delta_3 := \delta_7^2/2$ and $\rho_1 := \rho_5$.
\end{proofof}

\begin{proofof}{Lemma~\ref{lem:locallystarshaped}}
We assume $x \in K, B(y,r) \subseteq K$ and $\norm{x-y} \leq \lambda r$ with $0< r \leq \rho_2$ where the 
constant $\rho_2$ will determined in the course of the proof.
We need to show that for every $t \in [0,1]$ the point $(1-t)x + ty \in K$.
Or, equivalently, $\varphi_x( (1-t)x + ty ) \in \Haa^d$ 
with $\varphi_x$ as provided by Lemma~\ref{lem:boundedderivatives} (assuming we have chosen
$\rho_2$ so that $\lambda \rho_2 < \rho_4$.)
For notational convenience we will write 

$$\psi(t) := \varphi_x( (1-t)x + ty ), \quad \quad 
\hat{\psi}(t) := \varphi_x(x) + (D\varphi_x)_x \cdot t(y-x), $$

\noindent
where $(D\varphi_x)_z$ denotes the matrix of first derivatives of
$\varphi_x$ evaluated at $z$.
By the multivariate Taylor theorem:

\begin{equation}\label{eq:hatdist} 
\norm{ \psi(t) - \hat{\psi}(t) }  
\leq \alpha \cdot \delta_6  \cdot t^2 \cdot \norm{y-x}^2 \leq \alpha \cdot \delta_6 \cdot \lambda^2 \cdot t^2 \cdot r^2, 
\end{equation}

\noindent
where $\delta_6$ is as given by Lemma~\ref{lem:boundedderivatives} and $\alpha=\alpha(d)$ is a constant that depends only
on the dimension $d$. 
By Corollary~\ref{cor:ballcor} we have 

$$ B(\psi(1), \delta_7 r ) = B( \varphi_x(y), \delta_7 r ) \subseteq \varphi_x[ B(y,r) ] \subseteq \Haa^d. $$

Assuming $r < \delta_7 / (2 \alpha \delta_6 \lambda^2)$, we have 
$\norm{ \psi(1) - \hat\psi(1) } < (\delta_7/2) \cdot r$ and hence

$$ B( \hat\psi(1), (\delta_7/2) \cdot r ) \subseteq \Haa^d. $$

Since also $\psi(0) = \varphi_x(x) \in \Haa^d$, we have 

$$ \conv( \{\psi(0) \} \cup B( \hat\psi(1), (\delta_7/2)\cdot r ) ) \subseteq \Haa^d. $$

In particular

$$ \{ (1-t)\psi(0) + t z : z \in B( \hat\psi(1), (\delta_7/2)\cdot r ) \} 
= B( \hat\psi(t), (\delta_7/2) \cdot t  r ) \subseteq \Haa^d. $$

By~\eqref{eq:hatdist} and $r \leq \rho_2$, having chosen the constant $\rho_2$ sufficiently small, we have
$\psi(t) \in B( \hat\psi(t), (\delta_7/2) \cdot t r)$.
In other words, we have established that $\varphi( (1-t) x + t y ) \in \Haa^d$ as required, for 
a suitable choice of the constant $\rho_2(K,\lambda)$.
\end{proofof}

\section{Discussion}

\subsection{A connection with the game of cops and robbers}

The game of cops and robbers is a two-player game played on a graph.
There are two players, a cop and a robber, each located on one of the vertices of the graph. 
Before the game begins, the cop chooses his starting vertex and after that the robber chooses his (he is allowed to 
take into account where the cop is when choosing his starting position). 
In each turn, the cop either stays put or moves to a vertex adjacent to a vertex he is currently on. 
After the cop has made his move the robber does the same, and the next turn starts.
The goal of the cop is to capture the robber, i.e.~to be located on the same vertex as the robber. 
If a graph $G$ is such that the cop has a strategy so that no matter how the robber plays, 
the cop is guaranteed to achieve his goal in a finite number of moves, then $G$ is cop-win.
Aigner and Fromme~\cite{AignerFromme84}, and independently Quilliot~\cite{quilliot}, have shown that a graph is cop-win if and only if it 
can be reduced to a single vertex via a sequence of deletions as in Lemma~\ref{lem:del}. Thus:

\begin{corollary}\label{cor:cop}
If $G$ is cop-win then $\Kcal(G)$ is contractible. 
\end{corollary}

In particular, $\Rcal(n,r)$ is contractible when 
the corresponding random geometric graph (its 1-skeleton) is cop-win.
It is however known~\cite{RGGcops} that the random geometric graph in two dimensions is not cop-win 
a.a.s.~for all $r \leq c \ln n / \sqrt{n}$, for some constant $c>0$. 
(The result in~\cite{RGGcops} is stated only for the uniform distribution
on the unit square, but the proof easily adapts with very little modification to the current setting restricted to 
two dimensions.) 
Note that the stated bound is a multiplicative factor $\Omega(\sqrt{\ln n})$ larger than the 
bound $O( \sqrt{\ln n / n } )$ we have established in Theorem~\ref{thm:contract} for contractibility in two dimensions.
Our proofs did however make use of ideas from~\cite{RGGcops} that were used to show
the random geometric graph is a cop-win for $r = \Omega( \left(\ln n / n\right)^{1/5} )$ in two dimensions.
(An alternative proof was given by Alon and Pralat~\cite{AlonPralat}.)
Part of our proof was to show that the random geometric graph is ``locally cop-win'' when 
$r \geq \text{const} \cdot \left(\ln n/n\right)^{1/d}$.

\subsection{Suggestions for further work}

In the light of several results on random geometric graphs (e.g.~\cite{hamiltonrgg, PenroseMST, Penrosekconn}) and 
random \v{C}ech complexes (e.g.~\cite{BobrowksiOliveira}) in the regime when $r = \Theta\left( (\ln n/n)^{1/d} \right)$ 
it seems natural to expect a ``sharp threshold'' for contractibility.

\begin{conjecture}
If $K \subseteq \eR^d$ is a convex body with smooth boundary and the distribution is uniform on $K$ then 
there exists $c=c(K)$ such that, for every fixed $\eps>0$:

$$ \Pee( \Rcal(n,r_n) \text{ is contractible} ) \xrightarrow[n\to\infty]{} 
\begin{cases} 
0 & \text{ if $r_n \leq (c-\eps)\left(\frac{\ln n}{n}\right)^{1/d}$ }, \\
1 & \text{ if $r_n \geq (c+\eps)\left(\frac{\ln n}{n}\right)^{1/d}$. }
\end{cases}
$$
\end{conjecture}

In fact we would expect this to hold without the assumption that the boundary of $K$ is smooth and under more 
general assumptions on the probability measure, but we have opted for this version to increase the 
chances of success for whoever chooses to attempt the problem.

Having another look at the proof of Lemma~\ref{lem:del}, we see that it is in fact shown that 
$\Kcal(G\setminus v)$ is a strong deformation retract 
of $\Kcal(G)$ if the closed neighbourhood of $v$ is contained in the closed neighbourhood of some other vertex. 
What is more, one can check that the operation of removing $v$ can be described as a sequence of collapses. 
So in fact when $G$ is cop-win then $\Kcal(G)$ is {\em collapsible}. 
(The definitions of strong deformation retract and collapse can for instance be found in Section 6.4 of~\cite{Kozlovbook}.) 
As mentioned in the previous section, there is a choice of $r_n \gg \left(\frac{\ln n}{n}\right)^{1/d}$ 
such that the random geometric graph $G(n,r_n)$ a.a.s.~is not cop-win.
This leads us to the following conjecture.

\begin{conjecture}
Suppose $K \subseteq \eR^d$ is a convex body with smooth boundary and the distribution is uniform on $K$.
There exists a sequence $\rho_n \gg \left(\frac{\ln n}{n}\right)^{1/d}$ such that 
$\Rcal(n,r_n)$ is a.a.s.~not collapsible whenever $r_n \leq \rho_n$.
\end{conjecture}

Another natural direction for further research would be to consider a setup where the random points are not all contained in $K$, 
but there is a small amount of random noise added to the point locations. 
That is, we initially choose each point randomly from a compact, smooth $d$-manifold with 
boundary but then we add a small i.i.d.~``error'' to it. It would be interesting to determine under which conditions 
on the noise our Theorems~\ref{thm:contract} and~\ref{thm:smooth} are still valid.
Previous work in this spirit includes~\cite{Niyogi2011} and~\cite{BobrowskiMukherjeeTaylor}.
Of course when the noise comes from a fixed distribution with unbouded support (e.g.~a multivariate standard normal) then there
is no chance of direct analogues of Theorems~\ref{thm:contract} and~\ref{thm:smooth} being valid due to ``outliers'' 
(see~\cite{Adlercrackle}). 
Instead, it makes sense to consider the fairly natural case where the co-variance matrix of the noise decays as a function of $n$, 
and characterize the behaviour in terms of the rate of decay, perhaps obtaining a kind of threshold result.

\subsection*{Acknowledgement} 

We thank Gert Vegter for helpful discussions. We thank the anonymous referees for comments that have improved our paper.

%

\bibliographystyle{hplain}
\bibliography{ReferencesMullerStehlik2}

\end{document}